\RequirePackage{fix-cm}
\documentclass[smallextended,natbib,envcountsect]{svjour3}       
\smartqed  
\usepackage{graphicx}
\usepackage{newtxtext,newtxmath}
\usepackage{amsmath}  
\usepackage{mathtools}  
\usepackage{bbm}  
\usepackage[short]{optidef}  
\usepackage{tikz}  
\usepackage{enumitem}  
\usepackage[colorlinks]{hyperref}  
\usepackage[nameinlink,capitalise]{cleveref}  
\usepackage{csquotes}
\makeatletter
\if@cref@capitalise
\crefname{problem}{Problem}{Problems}
\crefname{observation}{Observation}{Observations}
\else
\crefname{problem}{problem}{problems}
\crefname{observation}{observation}{observations}
\fi
\makeatother
\Crefname{problem}{Problem}{Problems}
\Crefname{observation}{Observation}{Observations}

\spdefaulttheorem{observation}{Observation}{\bfseries}{\rmfamily}

\tikzset{
	chosen/.style={
		rectangle,
	}
}

\newcommand{\ie}{that is, }
\newcommand{\eg}{for example\ }
\newcommand{\crefpart}[2]{\hyperref[#2]{\cref*{#1}~\labelcref*{#2}}}  

\newcommand{\neighbour}{neighbour}
\newcommand{\labell}{labell}
\newcommand{\dualiz}{dualis}

\newcommand{\myvec}[1]{\boldsymbol{#1}}  
\newcommand{\ZZ}{\mathbb{Z}}  

\newcommand{\sumV}{\sum_{v\in V(G)}}  

\DeclareMathOperator{\nb}{N}
\newcommand{\N}[1]{\nb(#1)}  
\newcommand{\Nc}[1]{\nb[#1]}  
\newcommand{\Ng}[2]{\nb_{#1}(#2)}
\newcommand{\Ncg}[2]{\nb_{#1}[#2]}
\newcommand{\subtreerooted}[2]{#1#2}  
\newcommand{\Tv}{\subtreerooted{T}{v}}

\newcommand{\pn}{private \neighbour}
\newcommand{\pnabbr}{}

\newcommand{\RDF}{Roman dominating function}
\newcommand{\RDFabbr}{}

\newcommand{\RDN}{\gamma_R}
\newcommand{\Rdn}{Roman domination number}
\DeclareMathOperator{\limited}{L}
\newcommand{\DRDN}{\limited_2}  
\newcommand{\DRdname}{\(2\)-limited packing}  
\newcommand{\DRdproblem}{\DRdname{} problem}
\newcommand{\DRdshort}{\DRdproblem}  
\newcommand{\DRdn}{\DRdname{} number}
\newcommand{\DRdset}{\DRdname{}}
\newcommand{\DRDSet}{B}
\DeclareMathOperator{\DRDP}{2LP}  
\DeclareMathOperator{\RDP}{RD}  
\DeclareMathOperator{\IP}{IP}
\DeclareMathOperator{\LP}{LP}

\ExplSyntaxOn
\NewDocumentCommand{\instring}{mmmm}
{
	\oleks_instring:nnnn { #1 } { #2 } { #3 } { #4 }
}
\tl_new:N \l__oleks_instring_test_tl
\cs_new_protected:Nn \oleks_instring:nnnn
{
	\tl_set:Nn \l__oleks_instring_test_tl { #1 }
	\regex_match:nnTF { \u{l__oleks_instring_test_tl} } { #2 } { #3 } { #4 }
}
\ExplSyntaxOff
\def\setdelimiter{\colon\,}
\def\replacecolonbydelimiter#1:#2\relax{#1\setdelimiter#2}
\def\setcontents#1{\replacecolonbydelimiter#1\relax}
\NewDocumentCommand{\Set}{m}{
	\instring{:}{#1}{
		\ensuremath{\left\{\setcontents{#1}\right\}}%
	}{
		\ensuremath{\left\{#1\right\}}%
	}%
}
\newcommand{\abs}[1]{\ensuremath{\lvert#1\rvert}}
\newcommand{\ipforall}{\forall}

\journalname{Journal of Combinatorial Optimization}
\usepackage{orcidlink} 
\begin{document}
	
	\title{A dual view of Roman domination: The \DRdproblem
	}
	
	
	\author{
		Oliver Bachtler~\orcidlink{0000-0001-7942-0750} \and Sven O.~Krumke~\orcidlink{0000-0002-8726-9963} \and Helena Weiß~\orcidlink{0009-0003-2825-6665}
	}
	
	\authorrunning{
		O.~Bachtler, S.~O.~Krumke, and H.~Weiß
	} 

	\institute{O.\ Bachtler 
		\and
		S.\ O.\ Krumke 
		\and
		H.\ Weiß \at
		RPTU University Kaiserslautern-Landau, \\Department of Mathematics, Kaiserslautern, Germany \\
		\and
		H.\ Weiß \at
		\email{helena.weiss@math.rptu.de}           
	}
	
	\date{Received: date / Accepted: date}

	\maketitle
	
	\begin{abstract}
		We consider the \DRdproblem{}: 
		for a graph \(G=(V,E)\) one seeks to find a maximum cardinality subset \(\DRDSet \subseteq V\), such that, for all \(v\in V\), the closed \neighbour hood of~\(v\) contains at most two vertices in~\(\DRDSet\). 
		We compare this packing problem to the well-known Roman domination problem by pointing out some similarities and differences in the behaviour of the optimal solutions of both problems and show that these two problems are weakly dual.
		
		We show that for trees, the two problems are strongly dual, letting us solve the Roman domination problem by computing an optimal solution to the \DRdproblem.
		
	\end{abstract}
	
	\section{Introduction}
	The Roman domination problem (RDP) has widely been studied since the early 2000s.
	Based on historical events, it developed to a mathematical problem \citep{Stewart1999, ReVelle2000, Cockayne2004}.
	\cite{Cockayne2004} show numerous properties for Roman dominating functions and first bounds on the Roman domination number \(\RDN\), \eg depending on the maximum degree \(\Delta\) of a graph:
	\(\RDN \geq \frac{2n}{\Delta + 1}\).
	They also compare it to the domination number $\gamma$ and observe that \(\gamma \leq \RDN \leq 2\gamma\).
	An upper bound on the Roman domination number is given by \cite{Chambers2009} with \(\RDN \leq \frac{4n}{5}\).
	They also found an upper bound depending on the maximum degree \(\RDN \leq n - \Delta + 1\).
	
	In \cite{Dreyer2000} it is shown that the decision version of the RDP is NP-complete in general and it is mentioned that there are proofs that it is NP-complete, even when restricted to chordal, bipartite, split, or planar graphs.
	In his thesis, \cite{Dreyer2000} also gives a linear-time algorithm to solve the RDP on trees and in their paper, \cite{PengTsai2007} show that the problem is solvable in linear time on graphs of bounded tree-width.
	Furthermore, \cite{Liedloff2008} prove that there are linear-time algorithms when the input is restricted to interval graphs or cographs.
	Moreover, a lot of variants on the RDP have been investigated, see \citet{Chellali2020survey,Chellali2021,Chellali2020} for an overview.
	
	In their paper \cite{Gallant2010} introduce \(k\)-limited packings in graphs and show some bounds for them. 
	For a graph~\(G\) and an integer~\(k\), a \emph{\(k\)-limited packing} is a subset of vertices \(B\subseteq V(G)\) such that \(\Nc{v} \cap B \leq k\) for all \(v\in V(G)\).
	The \(k\)-limited packing number is the maximum cardinality of a \(k\)-limited packing.
	Comparing the \(k\)-limited packing number to \(k\)-tuple dominating sets, they show that \(\limited_{r-k}(G) + \gamma_{\times(k+1)}(G) = \abs{V(G)}\).
	Also they get that for the case \(k=2\) it holds \(\DRDN(G) \leq \frac{4}{5}\abs{V(G)}\) which also follows by \cite{Chambers2009} and weak duality to the RDP which we show in this paper.
	\cite{Gallant2010} also characterise trees whose \(k\)-limited packing number equals twice the domination number.
	Together with our result of strong duality on trees, it follows that this characterisation equals the one for Roman trees in \cite{Henning2002}.
	\cite{Dobson2011} then showed that the problem of finding \(k\)-limited packings is NP-complete while the work of \cite{Telle1993} implies that the problem is solvable in linear time on graphs of bounded tree-width.
	Thus, this problem is very similar to the RDP in a computational sense.
	%
	
	Integer programming together with linear programming duality can be used to derive combinatorial inequalities \citep{Nemhauser1988}.
	A classical example is to relax an IP formulation of the maximum matching problem, whose dual, after reinstating the integrality constraints, is the vertex cover problem.
	This shows that a minimum vertex cover contains at least as many vertices as a maximum matching has edges.
	Another example is regarded by \cite{Bachtler2024}, who show weak duality of the almost disjoint path problem and the separating by forbidden pairs problem.
	Another classical example, where even strong duality holds, is the integer maximum flow and minimum cut problem \citep{Ahuja1993}.
	
	In this paper, we \dualiz e the Roman domination problem and by this, show weak duality to the \(2\)-limited packing problem.
	It turns out there are some similarities between the RDP and the \DRdshort{} like the complexity in general and on trees.
	Our main result is \cref{strongdualitytrees}, which states that on trees the solutions of the RDP and the \DRdshort{} always coincide which implies that the corresponding linear relaxations of the integer linear programs always have an integral solution for trees.
	We also point out differences between these problems by checking whether known results for the Roman domination problem also hold for the \DRdproblem.
	In particular, we show that the duality gap for these two problems is unbounded, even in a multiplicative sense.
	
	In the next section, we give an overview of important notation and state some properties corresponding to the RDP.
	In \cref{sec:initial-properties}, we show weak duality of the \DRdproblem{} to RDP and show some basic properties that are analogous to known results for the RDP.
	\Cref{strongduality} contains our main result, the strong duality of our two problems on trees, before in the last section we conclude this work with some further ideas.
	\section{Preliminaries}
	\label{sec:prelims}
	
	All graphs considered here are finite, simple, and undirected.
	We denote the vertex and edge set of a graph~\(G\) by \(V(G)\) and \(E(G)\). 
	By \(\Ng{G}{v}\) we denote the (open) \neighbour hood of a vertex~\(v\) in~\(G\), \ie the set of vertices adjacent to~\(v\).
	By \(\Ncg{G}{v}\) we denote the closed \neighbour hood \(\Ncg{G}{v} \coloneqq \Ng{G}{v} \cup \Set{v}\).
	If the graph~\(G\) is clear from the context, we also write \(\N{v}\) and \(\Nc{v}\) for these sets.
	The degree~\(\deg(v)\) of vertex~\(v\) is the size of its \neighbour hood.
	The maximum degree of a graph \(G\) is denoted by~\(\Delta(G)\).
	For a subset \(W\subseteq V\) of vertices, \(G[W]\) is the subgraph induced by \(W\).
	A tree with a vertex designated as root is called a rooted tree.
	For a rooted tree~\((T,r)\) rooted at~\(r\) we denote the subtree rooted at~\(v\) by \(\Tv\).
	In particular, \(T=\subtreerooted{T}{r}\).
	We denote the path graph, the cycle graph, the complete graph, the empty graph on~\(n\) vertices, and the complete \(n\)-partite graph on \(m_1 + \dots + m_n\) vertices by \(P_n\), \(C_n\), \(K_n\), \(\bar{K}_n\), and \(K_{m_1,\dots,m_n}\) respectively.
	A \(K_{1,n}\) is called star graph.
	
	A \emph{\DRdset{}} in~\(G\) is a subset \(\DRDSet\subseteq V(G)\) with \(\abs{\Nc{v} \cap \DRDSet} \leq 2\). By \(\DRDN(G)\), we denote the \DRdn{} which is the size of a maximum cardinality \DRdset{}.
	
	A \emph{Roman dominating function}\RDFabbr{} for a graph~\(G\) is a map \(f\colon V(G)\to \Set{0,1,2}\) such that for each \(v\in V(G)\) with \(f(v) = 0\) there is a \neighbour{} \(u\) of~\(v\) with \(f(u) = 2\).
	We call \(\sumV f(v)\) the \emph{weight} of~\(f\).
	For a graph~\(G\) the minimum weight of a Roman dominating function is denoted by \(\RDN(G)\) and called \emph{Roman domination number} of~$G$. 
	We call a Roman dominating function with weight \(\RDN(G)\) a \(\RDN\)-function.
	In the Roman domination problem one seeks to compute~\(\RDN(G)\) for a graph~\(G\).
	
	A vertex~\(u\) is a \emph{private \neighbour} \pnabbr{}of~\(v\) with respect to a subset \(V_2\subseteq V(G)\) of vertices if \(u\in\Nc{v}\) but \(u\notin \Nc{v'}\) for all $v'\in V_2\setminus\Set{v}$. 
	A \pn{} of~\(v\) is \emph{external} if it is a vertex of \(V(G)\setminus V_2\).
	
	In \cite{Cockayne2004} a lot of properties for \RDF s were proven.
	For later reference we state some of them here.
	\begin{proposition}[Prop.~3 in \cite{Cockayne2004}]\label[proposition]{RDFproperties}
		For a graph~\(G\), let \(f\) be some \(\RDN\)-function and \(V_i = \Set{v\in V(G): f(v) = i}\) the vertices with value~\(i\).
		\begin{enumerate}[label=(\alph*)]
			\item \(\Delta(G[V_1]) \leq 1\).\label{deltaofones}
			\item There is no edge between a vertex in \(V_1\) and one in \(V_2\).\label{noonetwo}
			\item Each vertex in \(V_0\) is adjacent to at most two in \(V_1\).\label{fewonesatzero}
			\item Each vertex in \(V_2\) has at least two \pn s with respect to \(V_2\).\label{pnsoftwos}
			\item If~\(v\) is isolated in \(G[V_2]\) and has precisely one external \pn~\(w\) wrt.~\(V_2\), then \(\N{w}\cap V_1 = \emptyset\).\label{nooneforprivatezero}
		\end{enumerate}
	\end{proposition}
	
	Furthermore, we state some basic properties of the Roman domination problem here, which we check for \DRdproblem{} in the next section, to see whether they are also applicable there.
	\begin{observation}[Observation~7 in \cite{Rad2012}]\label[observation]{rdn-edge-removal}
		When deleting an edge, \(\RDN\) does not decrease.
	\end{observation}
	
	\begin{observation}[Proposition~7 and~8 in \cite{Cockayne2004}]\label[observation]{rdn-special-graphs}
		\begin{enumerate}[label=(\alph*)]
			\item \(\RDN(P_{n}) = \RDN(C_{n}) = \left\lceil\frac{2}{3}n\right\rceil\).
			\item Let \(m_1\leq m_2 \leq \dots \leq m_n\), then \(\RDN(K_{m_1,\dots,m_n}) = 2\) if \(m_1=1\), \(\RDN(K_{m_1,\dots,m_n}) = 3\) if \(m_1=2\), and \(\RDN(K_{m_1,\dots,m_n}) = 4\) if \(m_1\geq 3\).
		\end{enumerate}
	\end{observation}
	
	\begin{observation}[Observation~1 and~2 in \cite{Sumenjak2012}]\label[observation]{rdn-23}
		\begin{enumerate}[label=(\alph*)]
			\item If \(G\) is not isomorphic to~\(\bar{K}_2\), then \(\RDN(G) = 2\) if and only if \(\Delta(G) = \abs{V(G)}-1\).\label{rdn-2}
			\item If \(G\) is a connected graph, then \(\RDN(G) = 3\) if and only if \(\Delta(G) = \abs{V(G)}-2\).\label{rdn-3}
		\end{enumerate}
	\end{observation}
	
	
	\section{Derivation and basic properties}
	\label{sec:initial-properties}
	
	In this \lcnamecref{sec:initial-properties} we show the \DRdproblem{} to be weakly dual to the Roman domination problem.
	We also determine some properties of the \DRdproblem{} and compare them to the Roman domination problem.
	
	We start by recalling the \DRdproblem{} in its decision variant.
	
	\begin{problem}[\DRdproblem{} (decision variant)]\label{DRDP}
		For a graph~\(G\) and an integer~\(b\), is there a \DRdset{} \(\DRDSet\) such that \(\abs{\DRDSet} \geq b\)?
	\end{problem}
	In the following, we call vertices in a \DRdset{} \emph{chosen} or \emph{selected}.
	
	\paragraph{Weak duality.}
	To see that the \DRdproblem{} is, indeed, weakly dual to the Roman domination problem, we use the integer programming formulation found in \cite{PoureidiFathali2023}, determine the dual of its linear relaxation, and see that the integer program corresponding to this dual is the \DRdproblem.
	In particular, the size of any \DRdname{} of a graph $G$ is a lower bound on its Roman domination number.
	
	The IP in question is called RDP-ILP-2 in \cite{PoureidiFathali2023}:
	\begin{mini}
		{\myvec{x},\myvec{y}}{\sumV x_v + 2\sumV y_v}{}{}
		\addConstraint{x_v + \sum_{u\in\Nc{v}} y_u}{\geq 1}{\quad \ipforall v\in V(G)}
		\addConstraint{x_v,y_v}{\in \Set{0,1}}{\quad \ipforall v\in V(G)}.
	\end{mini}
	Here, \(x_v = 1\) represents \(f(v) = 1\) and \(y_v=1\) implies \(f(v) = 2\).
	If both are set to \(0\), then \(f(v) = 0\) as well.
	Thus, the objective coefficients are \(1\) for the \(x\)-variables and \(2\) for the \(y\)-variables.
	The constraints ensure that, for each vertex \(v\), either \(v\) is assigned a \(1\) or some neighbour in its closed \neighbour hood receives a \(2\).
	In particular, \(x_v=0\) if \(y_v=1\). 
	
	We now transition to the LP-relaxation by replacing $x_v,y_v\in\Set{0,1}$ by $x_v,y_v\geq 0$.
	Note that no optimal solutions will set the variables to values greater than one, letting us drop the $x_v,y_v\leq 1$ constraints.
	When we \dualiz e the resulting linear program, we get
	\begin{maxi}{\myvec{a}}{\sumV a_v}{\label[minmax]{RDdual}}{}
		\addConstraint{\sum_{u\in\Nc{v}} a_u}{\leq 2}{\quad \ipforall v\in V(G)}
		\addConstraint{a_v}{\leq 1}{\quad \ipforall v\in V(G)}
		\addConstraint{a_v}{\geq 0}{\quad \ipforall v\in V(G)}.
	\end{maxi}
	
	The integer version of this can be interpreted as choosing a maximum number of vertices such that at most two vertices are chosen in each closed \neighbour hood, which is exactly the \DRdproblem.
	Thus, this problem is, indeed, weakly dual to the Roman domination problem.
	
	Hence,
	\begin{theorem}
		Let~\(G\) be a graph. 
		Then \(\RDN(G)\geq \DRDN(G)\).
	\end{theorem}
	
	\paragraph{Basic properties of the \DRdproblem.}
	The \DRdproblem{} is a special case of a variant of the Multidimensional Knapsack Problem, called Linear Knapsack Problem (LKP) in \cite{Borgmann2023}.
	The LKP is formulated as follows:
	\begin{maxi*}{\myvec{x}}{\sumV a_v x_v}{}{}
		\addConstraint{\sum_{v\in F} x_v\leq \kappa_F}{\quad \ipforall F\in\mathcal{F}}
		\addConstraint{x_v\in\Set{0,1}}{\quad \ipforall v\in V(G)}
	\end{maxi*}
	for a graph \(G\), vertex weights \(a_v\in\ZZ_{\geq 0}\), and~\(\mathcal{F}\) a set of connected subgraphs of~\(G\) given with capacities \(\kappa_F\). 
	Setting the vertex weights to~\(1\) for all vertices of~\(G\) and \(\mathcal{F}\) to be the set of all closed \neighbour hoods (so each graph in \(\mathcal{F}\) is a star graph) with capacity~\(2\) for each graph in~\(\mathcal{F}\), we get the \DRdproblem{}.
	For general~\(\mathcal{F}\) \cite{Borgmann2023} showed that LKP is polynomial time solvable on path graphs but NP-hard on trees.
	
	There are some similarities between the \DRdproblem{} and the Roman domination problem. 
	For example, the analogue to \cref{rdn-edge-removal} is true.
	\begin{observation}
		When deleting an edge, \(\DRDN\) does not decrease.
	\end{observation}
	\begin{proof}
		Closed \neighbour hoods in the graph after deletion were also part of closed \neighbour hoods before deletion.\qed
	\end{proof}
	
	Next, let us take a look at some specific graph classes.
	\begin{observation}[Lemma~3 in \cite{Gallant2010}]\label[observation]{tlp-special-graphs}
		\begin{enumerate}[label=(\alph*)]
			\item \(\DRDN(P_n) = \left\lceil \frac{2}{3}n \right\rceil\).
			\item \(\DRDN(C_n) = \left\lfloor \frac{2}{3}n \right\rfloor\).
			\item \(\DRDN(K_{m,n}) = 2\) for \(m,n \geq 1\)
		\end{enumerate}
	\end{observation}
	%
	
	The last result is actually true for any complete \(n\)-partite graph \(K_{m_1,\dots,m_n}\):
	selecting three vertices in one part of the partition would yield a vertex with three selected \neighbour s, which cannot happen.
	Similarly, selecting two vertices in one part and one in another yields a selected vertex with two selected \neighbour s, which is equally impossible.
	The same would happen if for three different parts one vertex each was selected.
	
	\begin{observation}\label[observation]{drdn-complete-multipartite}
			Let \(m_1\leq m_2 \leq \dots \leq m_n\), then \(\DRDN(K_{m_1,\dots,m_n}) = 2\).
	\end{observation}
	
	When comparing \cref{tlp-special-graphs} to \cref{rdn-special-graphs}, we see that the Roman domination number and the \DRdn{} coincide on paths, stars, and cycles whose length is divisible by three.
	For other cycles \(C_n\), \(\RDN(C_n) = \DRDN(C_n) + 1\).
	For complete bipartite graphs that are not stars, the two numbers differ by one or two as well.
	In fact, the duality gap can get arbitrarily large: 
	regard for example the~\(C_4\) with \(\RDN(C_4) = 3\) and \(\DRDN(C_4) = 2\) and take the \(kC_4\), that is, the disjoint union of \(k\) copies of the \(C_4\).
	Since the \Rdn{} and the \DRdn{} are additive with respect to connected components, we have \(\RDN(kC_4)-\DRDN(kC_4) = k\).
	
	Thus, regarded additively, the duality gap is unbounded.
	This is also true when it is considered multiplicatively and for a connected graph, as we will now see.
	Let~$G_n$, \(n\geq 3\), be the graph that consists of a \(\bar{K}_n\) and a \(K_m\), where $m=\binom{n}{3}$.
	These two graphs are interconnected as follows:
	each vertex in the \(K_m\) corresponds to a triple of vertices in the \(\bar{K}_n\) and is connected to exactly the vertices in this triple by an edge.
	%
	%
	%
	
	All the graphs \(G_n\) have a \DRdn{} of~2.
	This is the case since we can always select two vertices and three are not possible here:
	let \(\DRDSet\) be a set of three vertices.
	By construction there exists at least one vertex \(v\) in the \(K_m\) that contains \(\DRDSet\cap \bar{K}_n\) in its \neighbour hood, and thus \(\DRDSet\) in its closed neighbourhood.
	
	The Roman domination number for these graphs is at least \(\frac{2}{3}n\) however.
	To see this, let \(f\) be a \RDF.
	If \(f(v)\neq 2\) for all \(v\in K_m\), then each of these \(m\) vertices is assigned a~\(1\) by \(f\) or it is adjacent to a vertex~\(w\in\bar{K}_n\) with \(f(w) = 2\).
	But such a vertex in \(\bar{K}_n\) has \(\binom{n-1}{2}\) \neighbour s in the \(K_m\), so we pay at least $\frac{4}{(n-1)(n-2)}$ per vertex, yielding weight at least \(\frac{2}{3}n\) for \(f\).
	
	On the other hand, if \(f(v)=2\) for some \(v\in K_m\), then all the vertices in the \(K_m\) have a \neighbour ing~\(2\).
	Consequently, we only need to deal with the vertices in $\bar{K}_n$.
	Here, no vertex needs to be be assigned a~\(2\), since they only have \neighbour s in \(K_m\).
	Setting \(f(w) = 1\) for a vertex \(w\in\bar{K}_n\) covers exactly that vertex and setting \(f(v) = 2\) for \(v\in K_m\) covers three vertices in \(\bar{K}_n\), so we pay at least \(\frac{2}{3}\) per vertex covered, yielding weight at least \(\frac{2}{3}n\) for \(f\) in this case as well.
	
	We complete this section by seeing whether the analogous statements of \cref{rdn-23} hold for the \DRdproblem{} as well.
	Similarly to \crefpart{rdn-23}{rdn-2}, we get
	\begin{observation}\label[observation]{deltanminusone}
		If \(G\) is a graph with \(\abs{V(G)}\geq 2\) and \(\Delta(G) = \abs{V(G)}-1\), then \(\DRDN(G) = 2\).
	\end{observation}
	\begin{proof}
		Let \(v\in V(G)\) with \(\deg(v) = \abs{V(G)}-1\).
		Then \(\Nc{v} = V(G)\) and can contain only two chosen vertices.\qed
	\end{proof}
	But in contrast to the Roman domination problem, there are graphs~\(G\) with \(\DRDN(G) = 2\) and \(\Delta(G) < \abs{V(G)}-1\), like the utility graph~\(K_{3,3}\), which satisfies \(\DRDN(K_{3,3}) = 2\) by \cref{tlp-special-graphs}.
	
	We also do not obtain an analogous result to \crefpart{rdn-23}{rdn-3}.  
	There are graphs~\(G\) with \(\Delta(G) = \abs{V(G)}-2\) and \(\DRDN(G) = 2 \neq 3\), for example \(K_{2,4}\) as well as graphs with \(\DRDN(G) = 3\) but \(\Delta(G) < \abs{V(G)}-2\), for example the \(C_5\), by \cref{tlp-special-graphs}.
	\section{Strong duality for trees}\label{strongduality}
	Now, we want to show that for trees we actually have equality between the \Rdn{} and the \DRdn.
	
	\begin{theorem}\label{strongdualitytrees}
		For trees~\(T\) we have \[\RDN(T) = \DRDN(T)\ .\]
	\end{theorem}
	
	To prove this, in the following we assume~\(T\) to be rooted at some vertex~\(r\).
	The idea of the proof is that given a minimum \RDF{}, we construct a \DRdset{} of the same size as the weight of the \RDF{}.
	We start with the following \lcnamecref{fancyRDFtrees}.
	
	\begin{lemma}\label[lemma]{fancyRDFtrees}
		Let~\(T\) be a rooted tree.
		There exists a \(\RDN\)-function~\(f\) such that, for all \(v\in T\),
		\begin{enumerate}[noitemsep,label={(\arabic*)}]
			\item\label{fewones} \(\Nc{v}\cap \subtreerooted{T}{v}\) contains at most one vertex \labell ed~\(1\),
			\item\label{twopnsbelow} if \(f(v) = 2\), \(\Nc{v}\cap \subtreerooted{T}{v}\) contains at least two \pn s of~\(v\), and
			\item\label{notwomarkedbelowzero} if \(f(v) = 0\), \(\Nc{v}\cap \subtreerooted{T}{v}\) contains at most one vertex~\(u\) such that
			\begin{enumerate}
				\item \(f(u) = 1\) or
				\item \(f(u) = 2\) and~\(u\) has exactly one external \pn{}.
			\end{enumerate}
		\end{enumerate}
	\end{lemma}
	
	Before we prove this statement, we apply it to prove \cref{strongdualitytrees}.
	Using a \RDF{} with the properties in \cref{fancyRDFtrees}, we construct a \DRdset{}.
	Let \(G\) be a graph and a \RDF~\(f\) for~\(G\) as in \cref{fancyRDFtrees} be given.
	Let \(\DRDSet\) be the set of all vertices with label~\(1\) and add, for each vertex~\(v\) with label~\(2\), two of the \pn s of \(v\) in \(Tv\), preferring child vertices.
	We want to show that \(\DRDSet\) is a \DRdset{}.
	For \(v\in T\), we regard \(\Nc{v}\).
	\begin{enumerate}[wide,label=\emph{Case~\arabic*:}]
		\item \(f(v) = 1\).
		In the closed \neighbour hood \(\Nc{v}\) of~\(v\) there are only vertices with label~\(0\) by \crefpart{RDFproperties}{noonetwo} and \crefpart{fancyRDFtrees}{fewones}.
		None of the children of~\(v\) is in~\(\DRDSet\) as they have no parent with label~\(2\). Thus \(\abs{\Nc{v}\cap \DRDSet}\leq 2\).
		\item \(f(v) = 0\).
		Again, children with label~\(0\) are not in~\(\DRDSet\), as they do not have a parent with label~\(2\).
		Together with \crefpart{fancyRDFtrees}{notwomarkedbelowzero}, there is at most one child of~\(v\) in~\(\DRDSet\).
		If \(v\notin\DRDSet\), we have \(\abs{\Nc{v}\cap \DRDSet}\leq 2\). So, assume \(v\in\DRDSet\).
		This means~\(v\) is a \pn{} of its parent, call it~\(w\) and \(f(w) = 2\).
		It follows that there is no child of~\(v\) with label~\(2\).
		The only problem would be if~\(w\in\DRDSet\) (thus \(w\) is a \pn{} of itself) and there was a child of~\(v\) that has label~\(1\).
		But this cannot happen by \crefpart{RDFproperties}{nooneforprivatezero}.
		So, we again have \(\abs{\Nc{v}\cap \DRDSet}\leq 2\).
		\item \(f(v) = 2\).
		By construction \(\abs{\Nc{v}\cap \Tv\cap \DRDSet}\leq 2\).
		We claim that the parent~\(w\) of~\(v\) is not in~\(\DRDSet\).
		By \crefpart{RDFproperties}{noonetwo} \(f(w)\neq 1\).
		If \(f(w)\) was~\(2\), it would not be a \pn{} of itself and thus, is not in~\(\DRDSet\).
		If \(f(w)\) was~\(0\), the parent of~\(w\) would have label~\(2\) and thus, \(w\) is not a \pn{} of any vertex with label~\(2\).
		Again it follows that~\(w\) is not in~\(\DRDSet\).
		In total, we have \(\abs{\Nc{v}\cap \DRDSet}\leq 2\).
	\end{enumerate}
	
	With this we can now prove our main result \cref{strongdualitytrees}.
	
	\begin{proof}[of \cref{strongdualitytrees}]
		Let~\(T\) be a tree rooted at some vertex.
		Given a \(\RDN\)-function with the properties as in \cref{fancyRDFtrees}, we construct a set with the aforementioned procedure.
		By construction, this set is a \DRdset{} of weight~\(\RDN\) as we choose one vertex for each~\(1\) and two vertices for each~\(2\).
		This means \(\DRDN(T) \geq \RDN(T)\).
		Together with weak duality \(\RDN(T) \geq \DRDN(T)\) we get \[\RDN(T) = \DRDN(T)\ .\]\qed
	\end{proof}
	
	\begin{corollary}
		The linear programs of the \DRdproblem{} and the RDP have always an integral solution value for trees.
	\end{corollary}
	\begin{proof}
		Let \(z_{\RDP}^{\IP}\), \(z_{\RDP}^{\LP}\), \(z_{\DRDP}^{\LP}\), \(z_{\DRDP}^{\IP}\) be the optimal solution values for the integer and the linear program for the RDP, and the linear and integer program of the \DRdproblem{}, respectively.
		Then we have \(z_{\RDP}^{\IP} \geq z_{\RDP}^{\LP} \geq z_{\DRDP}^{\LP} \geq z_{\DRDP}^{\IP}\).
		By \cref{strongdualitytrees}, we have \(z_{\RDP}^{\IP} = z_{\DRDP}^{\IP}\) and thus equality instead of all the inequalities before, implying the statement.\qed
	\end{proof}
	
	\begin{proof}[of \cref{fancyRDFtrees}]
		First, take any \(\RDN\)-function~\(f\) of the rooted tree.
		Starting at the leaves we go through the vertices in a bottom-up fashion.
		For every vertex~\(v\) we check whether the properties hold in \(\subtreerooted{T}{v}\) and, if one of the properties is violated, we transform the \RDF{} \(f\) into a \RDF{} \(f'\) and prove that after the transformation all properties are met in \(Tv\).
		For notational convenience, we define \(V_i\coloneqq\Set{v\in V(G) : f(v) = i}\) and \(V_i'\coloneqq\Set{v\in V(G) : f'(v) = i}\) for \(i\in\Set{0,1,2}\).
		
		Now, let \(v\in T\) such that the properties hold for all other vertices in \(\subtreerooted{T}{v}\).
		\begin{enumerate}[wide,label=\emph{Case~\arabic*:}]
			\item \(f(v)=1\).
			Assume~\(v\) has a child~\(u\) with \(f(u) = 1\), as otherwise we are done.
			Let \(f'\) be the function with \(f'(v)\coloneqq 2\) and \(f'(u)\coloneqq 0\) (and all other values are identical to \(f\)), see \cref{neighbouringones}.
			The weight stays the same and all vertices in \(V_0'\) are adjacent to one in \(V_2'\), so~\(f'\) is also a \(\RDN\)-function.
			Now,~\(u\) and~\(v\) do not have any children in \(V_1'\) since they had none in \(V_1\) by \crefpart{RDFproperties}{deltaofones}.
			Since~\(u\) and~\(v\) had no \neighbour s in~\(V_2\) by \crefpart{RDFproperties}{noonetwo}, they are \pn s of \(v\) with respect to \(V_2'\) and, thus, \(f'\) satisfies all properties in~\(\subtreerooted{T}{u}\) and~\(\Tv\).
			\begin{figure}[htb]
				\centering
				\begin{minipage}{0.4\textwidth}
					\ \hfill
					\begin{tikzpicture}[
						every node/.style={draw,circle,minimum size=0.45cm}
						]
						\node[label=left:\(v\)] (1) at (0,0) {1};
						
						\node[label=left:\(u\)] (2) at (-0.5,-1) {1};
						
						\draw (2) -- (1);
					\end{tikzpicture}
				\end{minipage}
				\begin{minipage}{0.1\textwidth}
					\centering
					\(\longrightarrow\)
				\end{minipage}
				\begin{minipage}{0.4\textwidth}
					\begin{tikzpicture}[
						every node/.style={draw,circle,minimum size=0.45cm}
						]
						\node[label=left:\(v\)] (1) at (0,0) {2};
						
						\node[label=left:\(u\)] (2) at (-0.5,-1) {0};
						
						\draw (2) -- (1);
					\end{tikzpicture}
				\end{minipage}
				\caption{Transformation if there are two \neighbour ing vertices \labell ed \(1\)}\label{neighbouringones}
			\end{figure}
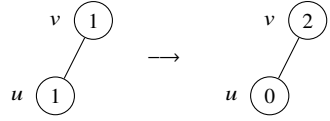
			
			\item \(f(v) = 2\).
			By \crefpart{RDFproperties}{noonetwo}, \(v\) has no children in~\(V_1\) and, by \crefpart{RDFproperties}{pnsoftwos}, \(v\) has at least two \pn s with respect to \(V_2\).
			If at least two of these are in \(Tv\), we are done.
			So assume at most one of these \pn s is in \(Tv\), then the parent of~\(v\), call it~\(w\), is one of its \pn s.
			
			If there is no \pn{} among the children of \(v\), then \(v\) is a \pn{} itself.
			Hence, no child of \(v\) is in~\(V_2\), so all children of \(v\) are in~\(V_0\).
			Additionally, since none of these children are \pn s of \(v\), they all have a child in \(V_2\).
			Now, we can swap the labels of \(v\) and \(w\) to obtain \(f'\), that is, \(f'(v)\coloneqq 0\) and \(f'(w)\coloneqq 2\), see \cref{nopnchild}.
			The function \(f'\) is also a \(\RDN\)-function and the vertex~\(v\) has no children in~\(V_1'\) or \(V_2'\), so the properties hold for~\(\Tv\).
			\begin{figure}[htb]
				\centering
				\begin{minipage}{0.4\textwidth}
					\ \hfill
					\begin{tikzpicture}[
						every node/.style={draw,circle,minimum size=0.45cm}
						]
						\node[label=left:\(w\)] (0) at (0.5,1) {0};
						
						\node[label=left:\(v\)] (1) at (0,0) {2};
						
						\node (2) at (-1,-1) {0};
						\node[draw=none] (dots) at (-0.5,-1) {\dots};
						\node (3) at (0,-1) {0};
						\node (4) at (1,-1) {2};
						
						\node (5) at (-1,-2) {2};
						\node (6) at (0,-2) {2};
						
						\draw (4) to node[midway,draw=none] {\rotatebox{-45}{\textsf{X}}} (1) -- (0);
						\draw (6) -- (3) -- (1) -- (2) -- (5);
					\end{tikzpicture}
				\end{minipage}
				\begin{minipage}{0.1\textwidth}
					\centering
					\(\longrightarrow\)
				\end{minipage}
				\begin{minipage}{0.4\textwidth}
					\begin{tikzpicture}[
						every node/.style={draw,circle,minimum size=0.45cm}
						]
						\node[label=left:\(w\)] (0) at (0.5,1) {2};
						
						\node[label=left:\(v\)] (1) at (0,0) {0};
						
						\node (2) at (-1,-1) {0};
						\node[draw=none] (dots) at (-0.5,-1) {\dots};
						\node (3) at (0,-1) {0};
						\node (4) at (1,-1) {2};
						
						\node (5) at (-1,-2) {2};
						\node (6) at (0,-2) {2};
						
						\draw (4) to node[midway,draw=none] {\rotatebox{-45}{\textsf{X}}} (1) -- (0);
						\draw (6) -- (3) -- (1) -- (2) -- (5);
					\end{tikzpicture}
				\end{minipage}
				\caption{Transformation if there are no children of \(v\) that are \pn s of \(v\)}\label{nopnchild}
			\end{figure}
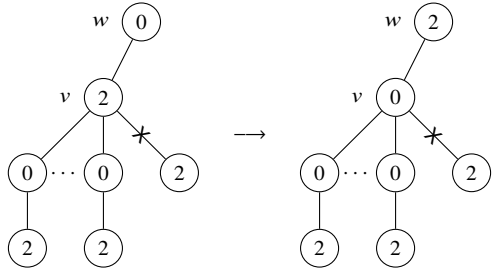
			
			Now, assume there is a \pn{} among the children of~\(v\), call it~\(u\).
			Then~\(u\) and~\(w\) are the only two \pn s of~\(v\) and~\(v\) itself is not a \pn{}, meaning that \(v\) has a child in \(V_2\).
			Also every other child of~\(v\) that is in \(V_0\) has a child in~\(V_2\) since it is not a \pn{} of~\(v\).
			We obtain a new function \(f'\) by setting \(f'(v)\coloneqq 0\) and \(f'(u)\coloneqq f'(w)\coloneqq 1\), see \cref{onepnchild}.
			The weight of \(f'\) and \(f\) coincide and the \neighbour s of~\(v\) in \(V_0'\) are still dominated by the reasons before and so is~\(v\).
			Hence, \(f'\) is a \(\RDN\)-function.
			Since \(f(v) = 2\), no children of \(v\) are in \(V_1\) and, thus, only one is in~\(V_1'\).
			All children of~\(v\) that are in~\(V_2\) have at least two external \pn s with respect to \(V_2'\) as they were not \pn s of themselves with respect to \(V_2\).
			Thus, the properties hold for~\(v\).
			It is, however, possible that~\(u\) has a child in~\(V_1'\).
			In this case we use the transformation described above for the case where \(f(v) = 1\) after which the properties hold for the entire~\(\Tv\).
			\begin{figure}[htb]
				\centering
				\begin{minipage}{0.4\textwidth}
					\ \hfill
					\begin{tikzpicture}[
						every node/.style={draw,circle,minimum size=0.45cm}
						]
						
						\node[label=left:\(w\)] (0) at (0.5,1) {0};
						
						\node[label=left:\(v\)] (1) at (0,0) {2};
						
						\node[label=left:\(u\)] (2) at (-1.5,-1) {0};
						\node (3) at (-0.5,-1) {0};
						\node[draw=none] (dots) at (0,-1) {\dots};
						\node (4) at (0.5,-1) {0};
						\node (5) at (1.5,-1) {2};
						
						\node (6) at (-1.5,-2) {2};
						\node (7) at (-0.5,-2) {2};
						\node (8) at (0.5,-2) {2};
						
						\draw (6) to node[midway,draw=none] {\rotatebox{-90}{\textsf{X}}} (2) -- (1) -- (3) -- (7);
						\draw (8) -- (4) -- (1) -- (0);
						\draw (5) -- (1);
					\end{tikzpicture}
				\end{minipage}
				\begin{minipage}{0.1\textwidth}
					\centering
					\(\longrightarrow\)
				\end{minipage}
				\begin{minipage}{0.4\textwidth}
					\begin{tikzpicture}[
						every node/.style={draw,circle,minimum size=0.45cm}
						]
						\node[label=left:\(w\)] (0) at (0.5,1) {1};
						
						\node[label=left:\(v\)] (1) at (0,0) {0};
						
						\node[label=left:\(u\)] (2) at (-1.5,-1) {1};
						\node (3) at (-0.5,-1) {0};
						\node[draw=none] (dots) at (0,-1) {\dots};
						\node (4) at (0.5,-1) {0};
						\node (5) at (1.5,-1) {2};
						
						\node (6) at (-1.5,-2) {2};
						\node (7) at (-0.5,-2) {2};
						\node (8) at (0.5,-2) {2};
						
						\draw (6) to node[midway,draw=none] {\rotatebox{-90}{\textsf{X}}} (2) -- (1) -- (3) -- (7);
						\draw (8) -- (4) -- (1) -- (0);
						\draw (5) -- (1);
					\end{tikzpicture}
				\end{minipage}
				\caption{Transformation if there is exactly one child of \(v\) that is a \pn{} of \(v\) and \(v\) is not a \pn{} of itself}\label{onepnchild}
			\end{figure}
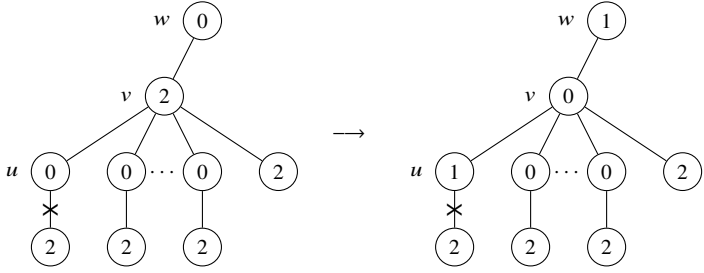
			\item \(f(v) = 0\).
			We look at different cases depending on which property is violated and how. 
			\begin{enumerate}[wide,label=\emph{Subcase~\arabic*:},labelindent=2\parindent] 
				\item Assume~\(v\) has exactly two \neighbour s \(v_1,\, v_2\) in~\(\Tv\) that are in~\(V_2\) and both have exactly one external \pn{} \(u_1,\, u_2\) with respect to \(V_2\).
				We transform the \RDF{} according to \cref{twomarkedtwoschildren}, \ie we set \(f'(v)\coloneqq 2\), \(f'(v_1)\coloneqq f'(v_2) \coloneqq 0\), and \(f'(u_1)\coloneqq f'(u_2)\coloneqq 1\).
				This has no effect on the weight and all vertices in \(V_0'\) have a neighbour in \(V_2'\).
				Hence, \(f'\) is a \(\RDN\)-function.
				
				None of the children of \(u_1\) or \(u_2\) are in~\(V_1\) by \crefpart{RDFproperties}{nooneforprivatezero}, and therefore none are in \(V_1'\).
				None of the children of \(v_1\) and \(v_2\) are in \(V_1\), by \crefpart{RDFproperties}{noonetwo}, nor are they in~\(V_2\) since they were \pn s of themselves.
				This makes them \pn s of \(v\) with respect to \(V_2'\).
				Lastly, none of the children of \(v\) are in~\(V_1'\) because \(f'\) is a \(\RDN\)-function and thus \crefpart{RDFproperties}{noonetwo} holds.
				So, the properties hold for the entire~\(\Tv\).
				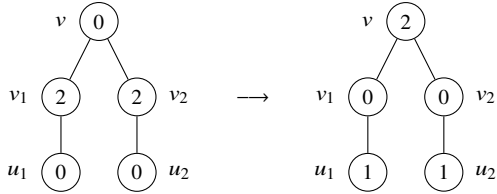
\begin{figure}[htb]
					\centering
					\begin{minipage}{0.4\textwidth}
						\ \hfill
						\begin{tikzpicture}[
							every node/.style={draw,circle,minimum size=0.45cm}
							]
							\node[label=left:\(v\)] (1) at (0,0) {0};
							
							\node[label=left:\(v_1\)] (2) at (-0.5,-1) {2};
							\node[label=right:\(v_2\)] (3) at (0.5,-1) {2};
							
							\node[label=left:\(u_1\)] (4) at (-0.5,-2) {0};
							\node[label=right:\(u_2\)] (5) at (0.5,-2) {0};
							
							\draw (4) -- (2) -- (1) -- (3) -- (5);
						\end{tikzpicture}
					\end{minipage}
					\begin{minipage}{0.1\textwidth}
						\centering
						\(\longrightarrow\)
					\end{minipage}
					\begin{minipage}{0.4\textwidth}
						\begin{tikzpicture}[
							every node/.style={draw,circle,minimum size=0.45cm}
							]
							\node[label=left:\(v\)] (1) at (0,0) {2};
							
							\node[label=left:\(v_1\)] (2) at (-0.5,-1) {0};
							\node[label=right:\(v_2\)] (3) at (0.5,-1) {0};
							
							\node[label=left:\(u_1\)] (4) at (-0.5,-2) {1};
							\node[label=right:\(u_2\)] (5) at (0.5,-2) {1};
							
							\draw (4) -- (2) -- (1) -- (3) -- (5);
						\end{tikzpicture}
					\end{minipage}
					\caption{\(f(v) = 0\) and there are exactly two children with value~\(2\) that have exactly one external \pn{}}\label{twomarkedtwoschildren}
				\end{figure}
				
				The same transformation shows that~\(v\) cannot have more than two such children since this would result in a \RDF{} with smaller weight contradicting the assumptions.
				\item Assume~\(v\) has a child \(v_1\) in~\(V_1\) and a child \(v_2\) in~\(V_2\) with exactly one external \pn{} \(u_2\).
				We transform the \RDF{} according to \cref{oneandmarkedtwochildren}, \ie we set \(f'(v)\coloneqq 2\), \(f'(v_1)\coloneqq f'(v_2)\coloneqq 0\), and \(f'(u_2)\coloneqq 1\).
				The weight of \(f\) and \(f'\) coincide and all vertices in~\(V_0'\) are adjacent to one in \(V_2'\), so we still have a \(\RDN\)-function.
				
				By the same reasons as before the properties hold for~\(v\), \(v_2\), and \(u_2\).
				For \(v_1\) the properties also hold since it has no children in~\(V_1\) or~\(V_2\) below it by \crefpart{RDFproperties}{noonetwo} and the properties for subgraphs of~\(\Tv\).
				\begin{figure}[htb]
					\centering
					\begin{minipage}{0.4\textwidth}
						\ \hfill
						\begin{tikzpicture}[
							every node/.style={draw,circle,minimum size=0.45cm}
							]
							\node[label=left:\(v\)] (1) at (0,0) {0};
							
							\node[label=left:\(v_1\)] (2) at (-0.5,-1) {1};
							\node[label=right:\(v_2\)] (3) at (0.5,-1) {2};
							
							\node[label=right:\(u_2\)] (4) at (0.5,-2) {0};
							
							\draw (2) -- (1) -- (3) -- (4);
						\end{tikzpicture}
					\end{minipage}
					\begin{minipage}{0.1\textwidth}
						\centering
						\(\longrightarrow\)
					\end{minipage}
					\begin{minipage}{0.4\textwidth}
						\begin{tikzpicture}[
							every node/.style={draw,circle,minimum size=0.45cm}
							]
							\node[label=left:\(v\)] (1) at (0,0) {2};
							
							\node[label=left:\(v_1\)] (2) at (-0.5,-1) {0};
							\node[label=right:\(v_2\)] (3) at (0.5,-1) {0};
							
							\node[label=right:\(u_2\)] (4) at (0.5,-2) {1};
							
							\draw (2) -- (1) -- (3) -- (4);
							
							\draw (2) -- (1) -- (3);
						\end{tikzpicture}
					\end{minipage}
					\caption{\(f(v) = 0\) and there is a child with value~\(1\) and one with value~\(2\) that has exactly one external \pn{}}\label{oneandmarkedtwochildren}
				\end{figure}
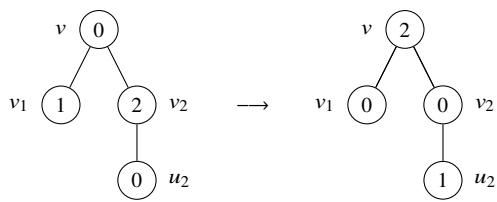
				\item Assume~\(v\) has two children \(v_1,\, v_2\) in~\(V_1\).
				We set \(f'(v)\coloneqq 2\) and \(f'(v_1)\coloneqq f'(v_2)\coloneqq 0\), see \cref{twoonechildren}.
				The weight of \(f'\) and \(f\) are the same and all vertices in~\(V_0'\) are adjacent to one in \(V_2'\).
				Thus, \(f'\) is a \(\RDN\)-function.
				
				Note, that~\(v\) cannot have more than two children in~\(V_1\), by \crefpart{RDFproperties}{fewonesatzero}, so no child of \(v\) is in \(V_1'\).
				Also, the vertices \(v_1\) and \(v_2\) are \pn s of \(v\) with respect to \(V_2'\).
				This is true since none of their children are in~\(V_2\) by \crefpart{RDFproperties}{noonetwo}.
				Since none of the children of \(v_1\) and \(v_2\) are in \(V_1\) by the properties for subgraphs of~\(\Tv\), these properties again hold for the entire~\(\Tv\).
				\begin{figure}[htb]
					\centering
					\begin{minipage}{0.4\textwidth}
						\ \hfill
						\begin{tikzpicture}[
							every node/.style={draw,circle,minimum size=0.45cm}
							]
							\node[label=left:\(v\)] (1) at (0,0) {0};
							
							\node[label=left:\(v_1\)] (2) at (-0.5,-1) {1};
							\node[label=right:\(v_2\)] (3) at (0.5,-1) {1};
							
							\draw (2) -- (1) -- (3);
						\end{tikzpicture}
					\end{minipage}
					\begin{minipage}{0.1\textwidth}
						\centering
						\(\longrightarrow\)
					\end{minipage}
					\begin{minipage}{0.4\textwidth}
						\begin{tikzpicture}[
							every node/.style={draw,circle,minimum size=0.45cm}
							]
							\node[label=left:\(v\)] (1) at (0,0) {2};
							
							\node[label=left:\(v_1\)] (2) at (-0.5,-1) {0};
							\node[label=right:\(v_2\)] (3) at (0.5,-1) {0};
							
							\draw (2) -- (1) -- (3);
						\end{tikzpicture}
					\end{minipage}
					\caption{Transformation if there is a vertex with two children \labell ed \(1\)}\label{twoonechildren}
				\end{figure}
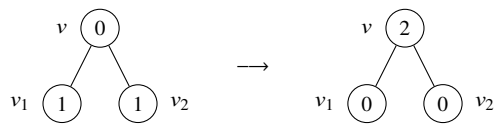
			\end{enumerate}
		\end{enumerate}
		
		Since we guaranteed that the three properties hold at all vertices in \(\subtreerooted{T}{v}\) and they trivially hold for leaves (which are not \labell ed~\(2\)), we can conclude inductively that the three properties hold for~\(T\).\qed
	\end{proof}
	
	\begin{remark}
		\Cref{strongdualitytrees} is not a characterization of graphs with \(\RDN = \DRDN\).
		There are other graphs than trees with \(\RDN = \DRDN\).
		For example all graphs with \(\Delta(G) = \abs{V(G)}-1\) as seen in \cref{deltanminusone}.
		Also see \cref{dualgapzero} for two examples without \(\Delta(G) = \abs{V(G)}-1\). 
		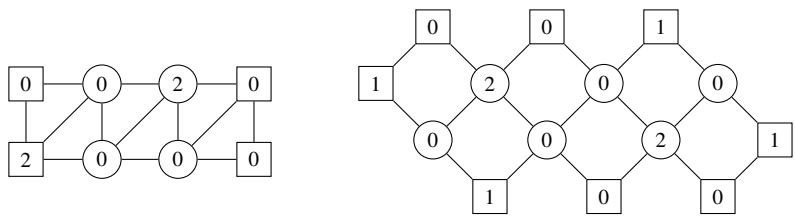
\begin{figure}[htb]
			\centering
			\begin{tikzpicture}[
				every node/.style={draw,circle,minimum size=0.45cm}
				]
				
				\node[chosen] (1) at (0,1) {0};
				\node	      (2) at (1,1) {0};
				\node         (3) at (2,1) {2};
				\node[chosen] (4) at (3,1) {0};
				
				\node[chosen]     (5) at (0,0) {2};
				\node             (6) at (1,0) {0};
				\node             (7) at (2,0) {0};
				\node[chosen] 	  (8) at (3,0) {0};
				
				\draw (1) -- (2) -- (3) -- (4);
				
				\draw (5) -- (6) -- (7) -- (8);
				
				\draw (2) -- (5);
				\draw (3) -- (6);
				\draw (4) -- (7);
				
				\draw (1) -- (5);
				\draw (2) -- (6);
				\draw (3) -- (7);
				\draw (4) -- (8);
				\node[draw=none] at (0,-0.5) {};
			\end{tikzpicture}
			\hspace{1cm}
			\begin{tikzpicture}[
				every node/.style={draw,circle,minimum size=0.45cm},
				scale=0.75
				]
				
				\node[chosen]  (1)  at (1,2) {0};
				\node[chosen]  (2)  at (3,2) {0};
				\node[chosen]  (3)  at (5,2) {1};
				
				\node[chosen]  (4)  at (0,1) {1};
				\node          (5)  at (2,1) {2};
				\node          (6)  at (4,1) {0};
				\node          (7)  at (6,1) {0};
				
				\node          (8)  at (1,0) {0};
				\node          (9)  at (3,0) {0};
				\node          (10) at (5,0) {2};
				\node[chosen]  (11) at (7,0) {1};
				
				\node[chosen]  (12) at (2,-1) {1};
				\node[chosen]  (13) at (4,-1) {0};
				\node[chosen]  (14) at (6,-1) {0};
				
				\draw (1) -- (4) -- (8) -- (12) -- (9) -- (13) -- (10) -- (14) -- (11) -- (7) -- (10) -- (6) -- (9) -- (5) -- (2) -- (6) --(3) -- (7);
				\draw (1)--(5)--(8);
			\end{tikzpicture}
			\caption{Example graphs with \(\RDN = \DRDN\). 
			The numbers in the vertices are the values of a Roman dominating function and the square vertices are those in the \DRdset.}\label{dualgapzero}
		\end{figure}
		For each graph there is a Roman dominating function and a \DRdset{} of the same weight given in the figure. 
		By duality, these weights are optimal and the left graph has \(\RDN = \DRDN = 4\) and the right one has \(\RDN = \DRDN = 8\).
	\end{remark}
	\section{Conclusion}
	We compared Roman domination and \DRdname{} and showed that the latter is weakly dual to the Roman domination problem. 
	Indeed, we could show strong duality on trees.
	
	In the literature there are many variants of the RDP.
	Like we did with the general problem one could \dualiz e these and see if these also generate new interesting problems. 
	Also a characterisation of graphs with \(\RDN = \DRDN\) would be of interest. 

	
	\paragraph{Funding} The work of Helena Weiß was partially funded by the Ministerium des Innern und für Sport, Rheinland-Pfalz within the OnePlan project.
	
	Oliver Bachtler was funded by the Deutsche Forschungsgemeinschaft (DFG, German Research Foundation) - GRK 2982, 516090167 \enquote{Mathematics of Interdisciplinary Multiobjective Optimization}
	
	\paragraph{Data availability} There is no data associated with this work.
	
	%
	\section*{Declarations}
	\paragraph{Conflict of interest} The authors declare that they have no conflict of interest.

	\bibliographystyle{plainnat}
	\bibliography{bibliography}   
	
	
\end{document}